\documentclass[12pt,a4paper,reqno]{amsart}
\allowdisplaybreaks
\usepackage{amsmath}
\usepackage{amsfonts}
\usepackage{amssymb,amsthm,amsfonts,amsthm,latexsym,enumerate,url,cases}
\numberwithin{equation}{section}
\usepackage{mathrsfs}
\usepackage{hyperref}
\hypersetup{colorlinks=true,citecolor=blue,linkcolor=blue,urlcolor=blue}
\def\pmod #1{\ ({\rm{mod}}\ #1)}

\theoremstyle{plain}
\newtheorem{theorem}{Theorem}
\newtheorem{lemma}{Lemma}

\newtheorem{conjecture}{Conjecture}
\theoremstyle{definition}

\newtheorem{remark}{Remark}

\usepackage{etoolbox}
\makeatletter
\patchcmd{\@settitle}{\uppercasenonmath\@title}{}{}{}
\patchcmd{\@setauthors}{\MakeUppercase}{}{}{}
\patchcmd{\section}{\scshape}{}{}{}
\makeatother

\begin{document}

\title
[{A new upper bound on Ruzsa's number on the Erd\H os--Tur\'{a}n conjecture}]
{A new upper bound on Ruzsa's number on the Erd\H os--Tur\'{a}n conjecture}

\author
[Y. Ding and L. Zhao] 
{Yuchen Ding \quad {\it and} \quad Lilu Zhao}

\address{(Yuchen Ding) School of Mathematical Sciences,  Yangzhou University, Yangzhou 225002, People's Republic of China}
\email{ycding@yzu.edu.cn}
\address{(Lilu Zhao) School of Mathematics, Shandong University, Jinan 250100, People's Republic of China}
\email{zhaolilu@sdu.edu.cn}


\begin{abstract}
In this note, we show that the Ruzsa number $R_m$ is bounded by $192$ for any positive integer $m$, which improve the prior bound $R_m\le 288$ given by Y.--G. Chen in 2008.
\end{abstract}
\maketitle

\section{Introduction}
For any positive integer $m$, let $\mathbb{Z}/m\mathbb{Z}$ be the set of residue classes mod $m$ (we represented $\mathbb{Z}/m\mathbb{Z}$ by $\mathbb{Z}_{m}$ for simplicity below). For any subsets $A,B\subseteq \mathbb{Z}_{m}$ and $n\in \mathbb{Z}_{m}$, let $\sigma_{A,B}(n)$ be the number of solutions to the equation $n=x+y~(x\in A,y\in B)$. Let $\sigma_{A}(n)=\sigma_{A,A}(n)$.
The Ruzsa number $R_m$ is defined to be the least positive integer $r$ so that there exists a set $A\subseteq \mathbb{Z}_{m}$ with $$1\le \sigma_{A}(n)\le r \quad (~\forall n\in \mathbb{Z}_{m}).$$
Our aim is to prove the following improvement of Chen's bound $R_m\le 288$ \cite{Chen1}. 

\begin{theorem}\label{thm1}
For any positive integer $m$ we have $R_m\le 192$.
\end{theorem}

Chen believed that $R_m\le 20$ for any positive integer $m$ (private communications).

Below, we briefly mention the motivation and history of Ruzsa's number. 

A subset $A$ of the natural numbers is called an asymptotic basis if $\sigma_{A}(n)\ge 1$ for all sufficiently large integers $n$.
In a cute article involving the Sidon set, Erd\H os and Tur\'{a}n \cite{Er-Tu} posed a well--known conjecture which states that if $A$ is an asymptotic basis, then the representation function $\sigma_{A}(n)$ cannot be bounded. People pay much attention to the Erd\H os--Tur\'{a}n conjecture. In 2003, Grekos, Haddad, Helou and Pihko \cite{GHHP} showed that for an asymptotic basis $A$, the representation function $\sigma_{A}(n)$ is not less than $6$ for infinitely many $n$. It was then improved to $\sigma_{A}(n)\ge 8$ infinitely many often by Borwein, Choi and Chu \cite{BCC}. In another approach, S\'{a}ndor \cite{Sandor} proved that if 
$$\limsup_{n\rightarrow\infty}\sigma_{A}(n)= K$$ for some given constant $K$, then $$\liminf_{n\rightarrow\infty}\sigma_{A}(n)\le K-2\sqrt{K}+1.$$
Chen \cite{Chen2} constructed an asymptotic basis $A$ of the natural numbers which satisfies that the set of $n$ with $\sigma_A(n)=2$ has density $1$. For more information on the Erd\H os--Tur\'{a}n conjecture, one can also refer to the nice books of Halberstam and Roth \cite{HalRo}, Tao and Vu \cite{Tao-Vu}.

Motivated by Erd\H os' question,  Ruzsa \cite{Ruzsa} constructed an asymptotic basis $A$ of natural numbers which has a bounded square mean value, i.e., 
$$\sum_{n\le N}\sigma_{A}^2(n)\ll N.$$
In his argument, Ruzsa actually proved that there is a given constant $C$ such that 
$R_m\le C$ for any positive integer $m$, which is surely a nontrivial result due to the consistent bound of all integers $m$. Using Ruzsa's construction, Tang and Chen \cite{Tang1} proved that $R_m\le 768$ for all sufficiently large $m$. Shortly after, they further \cite{Tang2} showed that $R_m\le 5120$ for any natural number $m$. In a previously beautiful work, Chen \cite{Chen1} improved the bound to $R_m\le 288$ for any natural number $m$. However, as Chen noted that he cannot improve the trivial lower bound $R_m\ge 3$ for $m\neq 1,3$. Recently, S\'{a}ndor and Yang \cite{SandorYang} give the nontrivial lower bound of $R_m$. They proved that $R_m\ge 6$ for $m\ge 36$. Moreover, S\'{a}ndor and Yang classified all the numbers $m$ with $R_m=k$ for $1\le k\le 5$ (see the Appendix of their article). In \cite{Chen3}, a related problem of the Ruzsa number was investigated by Chen and Sun. 

\section{Proofs}
Before presenting the proof of Theorem \ref{thm1}, we list a few lemmas to be used later.

\begin{lemma}\cite[Theorem 1]{Chen1}\label{lem1}
Let $p$ be a prime. Then $R_{2p^2}\le 48$.
\end{lemma}

\begin{lemma}\label{lem2}
For $m\le 4356$, we have $R_m\le 132$.
\end{lemma}
\begin{proof}
For $m\le 4356$, let 
$$A=\{0,1,2,\cdots,66,2\cdot 66,3\cdot66,\cdots,66\cdot66\}.$$
Then $\sigma_A(n)\ge 1$ for any $n\in \mathbb{Z}_{m}$ by the Euclidean division. It is clear that $$\sigma_A(n)\le |A|\le 132.$$

This completes the proof of Lemma \ref{lem2}.
\end{proof}

\begin{lemma}\cite[Lemma 4]{Chen1}\label{lem3}
For any real number $x\ge 33$, there exists at least one prime in the interval $\left(x,\frac{2x}{\sqrt{3}}\right]$.
\end{lemma}
\begin{proof}
Let $\pi(x)$ be the number of primes not exceeding $x$. We know, from the results of Panaitopol \cite{Panaitopol}, that
$$\pi(x)<\frac{x}{\log x-1-(\log x)^{-0.5}} \quad (x\ge 6)$$
and
$$\pi(x)<\frac{x}{\log x-1+(\log x)^{-0.5}} \quad (x\ge 59).$$
We first show that
$$\pi\left(\frac{2x}{\sqrt{3}}\right)-\pi(x)>\frac{\frac{2x}{\sqrt{3}}}{\log \frac{2x}{\sqrt{3}}-1+\left(\log \frac{2x}{\sqrt{3}}\right)^{-0.5}}-\frac{x}{\log x-1-(\log x)^{-0.5}}\ge0$$
for $x>1242$. By simple rearrangements of the inequality, it suffices to prove that
$$\left(1-\frac{\sqrt{3}}{2}\right)\log x\ge 1-\frac{\sqrt{3}}{2}+\frac{\sqrt{3}}{2}\log \frac{2}{\sqrt{3}}+(\log x)^{-0.5}+\frac{\sqrt{3}}{2}\left(\log \frac{2x}{\sqrt{3}}\right)^{-0.5}.$$
Checking by the computer, we know that it does satisfy for $x=1242$, and hence for $x>1242$.
For $33\le x\le1242$, our lemma can again be easily checked via help of the computer.

This completes the proof of Lemma \ref{lem3}.
\end{proof}

\begin{lemma}\label{lem4}
Let $m_1$ and $m_2$ be two positive integers with $\frac{3}{2}m_1\le m_2< 2m_1$. Then $$R_{m_2}\le 4R_{m_1}.$$
\end{lemma}
\begin{proof}
We follow the proof of Chen \cite[Lemma 3]{Chen1} with some refinements. Let $A\subseteq \mathbb{Z}_{m_1}$ be the one satisfying that 
$$1\le\sigma_{A}(n)\le R_{m_1}$$
for any $n\in \mathbb{Z}_{m_1}$. We make the provision that if $a\in A$, then $0\le a\le m_1-1$. Suppose that $m_2=m_1+r$, then we have $\frac{1}{2}m_1\le r< m_1$. Let
$$B=A\cup \{a+r:a\in A\}.$$

We first claim that $\sigma_{B}(n)\ge 1$ for any $n\in \mathbb{Z}_{m_2}$. In fact, for $0\le n\le m_1-1$, there exist $a_1,a_2\in A$ such that $n\equiv a_1+a_2\pmod{m_1}$, which means that $$n=a_1+a_2 \quad \text{or} \quad n+m_1=a_1+a_2$$ since $a_1+a_2\le 2m_1-2<2m_1$.
If $n=a_1+a_2$, then $n\equiv a_1+a_2\pmod{m_2}$. Otherwise, we have $n+m_1=a_1+a_2$, then $n+m_2=a_1+(a_2+r)$, from which we deduce that $n\equiv a_1+(a_2+r)\pmod{m_2}$. Thus, we have shown that $\sigma_{B}(n)\ge 1$ for any $0\le n\le m_1-1$. We now consider the case $m_1\le n\le m_2-1$. In this case, we have $$0< m_1-r\le n-r\le m_1-1.$$
Hence, there exist $a_1',a_2'\in A$ such that $n-r\equiv a_1'+a_2'\pmod{m_1}$, which implies that
$$n-r=a_1'+a_2' \quad \text{or} \quad n-r+m_1=a_1'+a_2'$$
again by the fact that $a_1'+a_2'<2m_1$. If $n-r=a_1'+a_2'$, then $n\equiv a_1'+(a_2'+r)\pmod{m_2}$. So, we can now assume that $n-r+m_1=a_1'+a_2'$, from which we deduce that $n+m_2=(a_1'+r)+(a_2'+r)$, i.e., $n\equiv (a_1'+r)+(a_2'+r)\pmod{m_2}$. Therefore, we have $\sigma_{B}(n)\ge 1$ for any $m_1\le n\le m_2-1$ too.

In the following, we prove that $\sigma_{B}(n)\le 4R_{m_1}$ for any $n\in \mathbb{Z}_{m_2}$. For any $n\in \mathbb{Z}_{m_2}$ with $0\le n\le m_2-1$, we can assume that
\begin{equation}\label{eq1}
n\equiv b_1+b_2\pmod{m_2}, \quad (b_1,b_2\in B).
\end{equation}
From Eq. (\ref{eq1}) we have 
\begin{equation}\label{eq2}
n+km_2=b_1+b_2, \quad b_1,b_2\in B
\end{equation}
for some nonnegative integers $k$. Since $b_1+b_2\le 2m_2-2<2m_2$, we know that $k=0$ or $1$. Thus, we can conclude that
\begin{equation}\label{eq3}
\sigma_{B}(n)\le \#\{(b_1,b_2):n=b_1+b_2\}+\#\{(b_1,b_2):n+m_2=b_1+b_2\},
\end{equation}
where the $b's$ belong to $B$ and $a's$ belong to $A$, respectively, from now on. By the definition of $B$ we have
\begin{align}\label{eq4}
\#\{(b_1,b_2):n=b_1+b_2\}\le&\#\{(a_1,a_2):n=a_1+a_2\}+2\#\{(a_1,a_2):n-r=a_1+a_2\}\nonumber\\
\quad &+\#\{(a_1,a_2):n-2r=a_1+a_2\}
\end{align}
and 
\begin{align}\label{eq5}
\#&\{(b_1,b_2):n+m_2=b_1+b_2\}\le\#\{(a_1,a_2):n+m_2=a_1+a_2\}\nonumber\\
\quad &+2\#\{(a_1,a_2):n+m_2-r=a_1+a_2\}+\#\{(a_1,a_2):n+m_2-2r=a_1+a_2\}\nonumber\\
&\quad \quad \quad\quad  \quad \quad \quad \quad \quad \quad \quad \quad 
\le\#\{(a_1,a_2):n+m_1+r=a_1+a_2\}\nonumber\\
\quad &+2\#\{(a_1,a_2):n+m_1=a_1+a_2\}+\#\{(a_1,a_2):n+m_1-r=a_1+a_2\}.
\end{align}
Below, we shall illustrate that there are some disjoint solutions in Eqs. (\ref{eq4}) and (\ref{eq5}) which would be the main novelty of this note. It will be separated into the following three cases.

{\bf Case I.} $\#\{(a_1,a_2):n-2r=a_1+a_2\}\neq0$. In this case, we have $n\ge 2r\ge m_1.$
It then follows that
$$n+m_1+r>n+m_1\ge 2m_1,$$
from which we deduce that
\begin{align}\label{eq6}
\#\{(a_1,a_2):n+m_1=a_1+a_2\}=\#\{(a_1,a_2):n+m_1+r=a_1+a_2\}=0
\end{align}
since $a_1+a_2\le 2(m_1-1)<2m_1$.
Hence, by Eqs (\ref{eq4}), (\ref{eq5}) and (\ref{eq6}) we obtain that
\begin{align*}
\sigma_{B}(n)\le \#&\{(a_1,a_2):n=a_1+a_2\}\nonumber\\
\quad &+2\#\{(a_1,a_2):n-r=a_1+a_2 \quad \text{or} \quad n+m_1-r=a_1+a_2\}
\nonumber\\
\quad &+\#\{(a_1,a_2):n-2r=a_1+a_2\}\nonumber\\
\le \#&\{(a_1,a_2):n\equiv a_1+a_2 \pmod{m_1}\}\nonumber\\
\quad &+2\#\{(a_1,a_2):n-r\equiv a_1+a_2 \pmod{m_1}\}
\nonumber\\
\quad &+\#\{(a_1,a_2):n-2r\equiv a_1+a_2 \pmod{m_1}\}\nonumber\\
\le 4&R_{m_1}.
\end{align*}

{\bf Case II.} $\#\{(a_1,a_2):n+m_1+r=a_1+a_2\}\neq0$. In this case, we have
$$n+m_1+r\le 2(m_1-1) \quad \text{or} \quad n\le m_1-2-r<\frac{1}{2}m_1$$
since $r\ge \frac{1}{2}m_1$. Then it follows that
$$n-2r<n-r<0,$$
from which we deduce that
\begin{align}\label{eq7}
\#\{(a_1,a_2):n-2r=a_1+a_2\}=\#\{(a_1,a_2):n-r=a_1+a_2\}=0.
\end{align}
Hence, by Eqs (\ref{eq4}), (\ref{eq5}) and (\ref{eq7}) we obtain that
\begin{align*}
\sigma_{B}(n)\le 2\#&\{(a_1,a_2):n=a_1+a_2  \quad \text{or} \quad n+m_1=a_1+a_2\}\nonumber\\
\quad &+\#\{(a_1,a_2):n+m_1+r=a_1+a_2\}
\nonumber\\
\quad &+\#\{(a_1,a_2):n+m_1-r=a_1+a_2\}\nonumber\\
\le 2\#&\{(a_1,a_2):n\equiv a_1+a_2 \pmod{m_1}\}\nonumber\\
\quad &+\#\{(a_1,a_2):n+r\equiv a_1+a_2 \pmod{m_1}\}
\nonumber\\
\quad &+\#\{(a_1,a_2):n-r\equiv a_1+a_2 \pmod{m_1}\}\nonumber\\
\le 4&R_{m_1}.
\end{align*}

{\bf Case III.} $\#\{(a_1,a_2):n-2r=a_1+a_2\}=\#\{(a_1,a_2):n+m_1+r=a_1+a_2\}=0$. In this last case, we clearly have
\begin{align*}
\sigma_{B}(n)\le 2\#&\{(a_1,a_2):n=a_1+a_2  \quad \text{or} \quad n+m_1=a_1+a_2\}\nonumber\\
\quad &+2\#\{(a_1,a_2):n-r=a_1+a_2  \quad \text{or} \quad n+m_1-r=a_1+a_2\}
\nonumber\\
\le 2\#&\{(a_1,a_2):n\equiv a_1+a_2 \pmod{m_1}\}\nonumber\\
\quad &+2\#\{(a_1,a_2):n-r\equiv a_1+a_2 \pmod{m_1}\}\nonumber\\
\le 4&R_{m_1}.
\end{align*}

This completes the proof of Lemma \ref{lem4}.
\end{proof}

\begin{remark} In \cite[Lemma 3]{Chen1}, Chen showed that $R_{m_1}\le 6R_{m_2}$ under the condition that $$m_1<m_2\le \frac{3}{2}m_1.$$ The improvement of the upper bound on the Ruzsa number comes from our refined relation between $R_{m_1}$ and $R_{m_2}$. The idea for this refined relation relies on the following two points. Firstly, we make the observation that if we use the basis of $\mathbb{Z}_{m_1}$ to generate a basis of $\mathbb{Z}_{m_2}$, then the parameter $k$ in Eq. (\ref{eq2}) can only be 0 or 1. Whereas, if we use the basis of $\mathbb{Z}_{m_2}$ to generate a basis of $\mathbb{Z}_{m_1}$, like the ones taken by Chen, then the corresponding parameter $k$ in Eq. (\ref{eq2}) could be 0 or 1 or 2, which makes 
relation between $R_{m_1}$ and $R_{m_2}$ less clean and involved.
Basing on this new construction, we introduce a new correlation of $m_1$ and $m_2$ (i.e., $\frac{3}{2}m_1\le m_2<2m_1$) to form some disjoint solutions to $\sigma_B(n)$ which certainly would be the key ingredient of our note.
\end{remark}

Now, let's turn to the proof of Theorem \ref{thm1}.
\begin{proof}[Proof of Theorem \ref{thm1}]
By Lemma \ref{lem2}, we only need to consider $m> 4356$.
By this condition we shall have $\sqrt{m/4}> 33$. By Lemma \ref{lem3}, there exists a prime $p$ such that
$$\sqrt{\frac{m}{4}}< p\le\sqrt{\frac{m}{3}}.$$
In other words, we have $$\frac{3}{2}\cdot2p^2=3p^2\le m< 4p^2=2\cdot2p^2.$$ Therefore, by Lemma \ref{lem1} and \ref{lem4} it follows that
$$R_{m}\le 4R_{2p^2}\le 192.$$

This completes the proof of Theorem \ref{thm1}.
\end{proof}

\section{Some related conjectures}
In \cite{Chen1}, Chen asked that whether there are infinitely many odd (resp. even) numbers in the sequence $\{R_m\}_{m=1}^{\infty}$. He called it a basic problem on Ruzsa's number. To end up this note, we pose a few conjectures which seems to be of some interests.

\begin{conjecture}\label{conjecture1}
For any positive integer $m$, we have $\left|R_{m+1}-R_m\right|\le 1.$
\end{conjecture}
For $m\le 34$, Conjecture \ref{conjecture1} is true by the table in the Appendix of the article of S\'andor and Yang \cite{SandorYang}. 
Perhaps, this conjecture is not really difficult. 

\begin{conjecture}[Chen]\label{conjecture2}
We have $R_{m+1}=R_m$ for infinitely many positive integers $m$.
\end{conjecture}

Comparing with Conjecture \ref{conjecture1}, it might be somewhat easier to reach Conjecture \ref{conjecture2}.

From now on, let $\mathcal{H}_m$ be the set of all subsets $A\subseteq\mathbb{Z}_{m}$ such that 
$\sigma_A(n)\ge 1$
for all $n\in \mathbb{Z}_{m}$. Define the following minimal mean value amount
$$\ell_m=\min_{A\subseteq\mathcal{H}_m}\left\{m^{-1}\sum_{n\in\mathbb{Z}_{m}}\sigma_{A}(n)\right\}.$$

\begin{conjecture}\label{conjecture3}
We have $\liminf_{m\rightarrow\infty}\ell_m\ge 3.$
\end{conjecture}

The philosophy leading to Conjecture \ref{conjecture3} is that we believe that for most integers $n\in \mathbb{Z}_{m}$ there are at least two essentially different representations, providing that $\sigma_A(n)\ge 1$
for all $n\in \mathbb{Z}_{m}$.
This can be viewed as an analogue of the Erd\H os--Tur\'{a}n conjecture in $\mathbb{Z}_{m}$. 
It is clear that we have $$\liminf_{m\rightarrow\infty}\ell_m\ge 2$$
from Lemma 2.2 of \cite{SandorYang}. By Theorem \ref{thm1}, we have 
$$\limsup_{m\rightarrow\infty}\ell_m\le 192.$$
The authors think that any improvement of the bounds involving $\liminf_{m\rightarrow\infty}\ell_m$ and $\limsup_{m\rightarrow\infty}\ell_m$ above would be of interest. Actually, Conjecture \ref{conjecture3} could be interpreted to the following more simple form: Let $A\subseteq\mathbb{Z}_{m}$ be a subset such that $A+A=\mathbb{Z}_{m}$, then
$$\liminf_{m\rightarrow\infty}\frac{|A|}{\sqrt{m}}\ge \sqrt{3}.$$
For positive integer $m$, let $\mathcal{H}^*_m$ be the set of all subsets of $\mathbb{Z}_{m}$ with $1\le \sigma_A(n)\le R_m$ for any $n\in \mathbb{Z}_{m}$.
Let $$\mathcal{K}_m=\min_{A\in \mathcal{H}^*_m}\{\#A\}.$$

\begin{conjecture}\label{conjecture4}
We have $\liminf_{m\rightarrow\infty}\mathcal{K}_m/\sqrt{3m}\ge 1.$
\end{conjecture}
Conjecture \ref{conjecture4} is slightly weaker than Conjecture \ref{conjecture3}. One may expect that $\ell_m$ is `always' attained for some $A\in \mathcal{H}^*_m$. That is,

\begin{conjecture}\label{conjecture5}
We have $\mathcal{K}_m^2=m\ell_m$ for sufficiently large $m$.
\end{conjecture}

\section*{Acknowledgments}
The authors would like to thank Professor Y.--G. Chen in Nanjing Normal University for pointing out some gaps in a former version of the manuscript. Also, many thanks to Doctors Guilin Li and Li--Yuan Wang for their generous help of the numerical calculations displayed in Lemma \ref{lem3} by computer.

The first named author is supported by National Natural Science Foundation of China  (Grant No. 12201544), Natural Science Foundation of Jiangsu Province, China (Grant No. BK20210784), China Postdoctoral Science Foundation (Grant No. 2022M710121), the foundations of the projects "Jiangsu Provincial Double--Innovation Doctor Program'' (Grant No. JSSCBS20211023) and "Golden  Phoenix of the Green City--Yang Zhou'' to excellent PhD (Grant No. YZLYJF2020PHD051).

The second named author is support by the National Key Research and Development Program of China (Grant No. 2021YFA1000700) and National Natural Science Foundation of China (Grant No. 11922113).

\end{document}